\documentclass[11pt,a4paper,oneside]{amsart}
\usepackage{fullpage}
\usepackage{amssymb,amsfonts}
\usepackage{epsfig}
\usepackage{graphicx}

\usepackage{graphics}
\newtheorem{theorem}{Theorem}[section]
\newtheorem{lemma}[theorem]{Lemma}
\newtheorem{corollary}[theorem]{Corollary}

\theoremstyle{definition}
\newtheorem{definition}[theorem]{Definition}

\numberwithin{equation}{section}

\DeclareMathOperator{\diam}{diam} 

\newcommand{\be}{\begin{equation}}
\newcommand{\ee}{\end{equation}}

\newcommand{\dist}{{\operatorname{dist}}}

\DeclareMathOperator{\rad}{rad}
\DeclareMathOperator{\supp}{supp}

\DeclareMathOperator{\capacity}{Cap}

\makeindex

\def\Xint#1{\mathchoice 
 {\XXint\displaystyle\textstyle{#1}}%
{\XXint\textstyle\scriptstyle{#1}}%
{\XXint\scriptstyle\scriptscriptstyle{#1}}%
 {\XXint\scriptscriptstyle\scriptscriptstyle{#1}}%
 \!\int}
\def\XXint#1#2#3{{\setbox0=\hbox{$#1{#2#3}{\int}$}
 \vcenter{\hbox{$#2#3$}}\kern-.5\wd0}}

 \def\dashint{\Xint-}

\usepackage[colorinlistoftodos]{todonotes}
\usepackage[colorlinks=true, allcolors=blue]{hyperref}

\title{Triebel-Lizorkin capacity and Hausdorff measure in metric spaces}
\author{Nijjwal Karak}
\address{Department of Mathematical Analysis, Charles University, Sokolovsk\'a 83, 18600 Prague 8, Czech Republic}
\email{nijjwal@gmail.com}
\thanks{This work was supported by OP RDE project no. CZ.02.2.69/0.0/0.0/16\_027/0008495, International Mobility of Researcher at Charles University.}
\begin{document}
\begin{abstract}
We provide a upper bound for Triebel-Lizorkin capacity in metric settings in terms of Hausdorff measure. On the other hand, we also prove that the sets with zero capacity have generalized Hausdorff $h$-measure zero for a suitable gauge function $h.$ 
\end{abstract}
\maketitle
\section{Introduction}
In this paper, we study the relation between Triebel-Lizorkin capacities in metric settings and the Hausdorff measure, see Theorem \ref{first-main} and Theorem \ref{second-main}. 
The definition of Triebel-Lizorkin capacity is taken from \cite{HKT17} and it is based on the pointwise definition of fractional $s$-gradients and Haj\l asz-Triebel-Lizorkin space $M^s_{p,q}(X),$ $0<s<\infty$ and $0<p,q\leq\infty,$ introduced in \cite{KYZ11}. The Triebel-Lizorkin capacity of a set $E\subset X$ is
\begin{equation*}
\capacity_{M^s_{p,q}}(E)=\inf\Big\{\Vert u\Vert_{M^s_{p,q}(X)}^p: u\in\mathcal{A}(E)\Big\},
\end{equation*}
where $\mathcal{A}(E)=\{u\in M^s_{p,q}(X): u\geq 1 \,\text{on a neighborhood of}\, E\}.$ The relation between Sobolev capacity and Hausdorff measure has been studied in \cite{BO05, KK15, KM96}, whereas the relation between Besov capacity and Hausdorff measure has been studied in \cite{Cos09, Net92, Nuu16}. We also give an estimate on the exceptional set of Lebesgue points of functions in fractional Sobolev space $W^{s,p}(X),$ $0<s,p<1,$ see Theorem \ref{third-main} and the discussion prior to that. 
\section{Notations and Preliminaries}
We assume throughout that $X=(X,d,\mu)$ is a metric measure space equipped with a metric $d$ and a Borel regular outer measure $\mu.$ 
A Borel regular measure $\mu$ on a metric space $(X,d)$ is called a \textit{doubling measure} if every ball in $X$ has positive and finite measure and there exist a constant $c_d\geq 1$ such that
$\mu(B(x,2r))\leq c_d\,\mu(B(x,r))$
holds for each $x\in X$ and $r>0.$ We call a triple $(X,d,\mu)$ a \textit{doubling metric measure space} if $\mu$ is a doubling measure on $X.$\\
\indent In this article, we will use the concept of $\gamma$-median instead of integral averages, introduced in \cite{PT12}. 
\begin{definition}
Let $0<\gamma\leq 1/2.$ The $\gamma$-median $m_u^{\gamma}(A)$ of a measurable, almost everywhere finite function $u$ over a set $A\subset X$ of finite measure is
 \begin{equation*}
m_u^{\gamma}(A)=\sup\{M\in\mathbb{R}:\mu(\{x\in A:u(x)<M\})\leq\gamma\mu(A)\}.
 \end{equation*}
 \end{definition}
We will mention here some basic properties of medians collected from \cite[Lemma 4.2 \& Remark 4.6]{Nuu16}.
\begin{lemma}
Let $0<\gamma\leq 1/2$ and $A\subset X.$ Then for two measurable functions $u,v:A\rightarrow\mathbb{R},$ the $\gamma$-median has the following properties:\\
$(a)$ If $\gamma\leq\gamma ',$ then $m_u^{\gamma '}\leq m_u^{\gamma},$\\
$(b)$ If $A\subset B$ and $\mu(B)\leq C\mu(A),$ then $m_u^{\gamma}(A)\leq m_u^{\gamma/C}(B),$\\
$(c)$ If $c\in\mathbb{R},$ then $m_u^{\gamma}+c=m_{u+c}^{\gamma},$\\
$(d)$ $\vert m_u^{\gamma}(A)\vert\leq m_{\vert u\vert}^{\gamma}(A),$\\
$(e)$ If $u\in L^p(A),$ $p>0,$ then
$$m_{\vert u\vert}^{\gamma}(A)\leq\left(\gamma ^{-1}\dashint_A\vert u\vert ^p\,d\mu\right)^{1/p},$$
$(f)$ If $u$ is continuous, then for every $x\in X,$
$$\lim_{r\rightarrow 0}m_u^{\gamma}(B(x,r))=u(x).$$
\end{lemma}
Let us recall the definition of Lebesgue points based on the medians.
 \begin{definition}
Let $0<\gamma\leq 1/2$ and $u$ be a measurable, almost everywhere finite function. A point $x\in X$ is a  \textit{generalized Lebesgue point} of $f,$ if
\begin{equation*}
\lim_{r\rightarrow 0}\,m_u^{\gamma}(B(x,r))=u(x).
\end{equation*}
\end{definition}
We use the pointwise definition of Triebel-Lizorkin space based on the fractional $s$-gradients, introduced in \cite{KYZ11}.
\begin{definition}\label{s-gradient}
Let $S\subset X$ be a measurable set and let $0<s<\infty.$ A sequence of nonnegative measurable functions $(g_k)_{k\in\mathbb{Z}}$ is a fractional $s$-gradient of a measurable function $u:S\rightarrow [-\infty,\infty]$ in $S,$ if there exists a set $E\subset S$ with $\mu(E)=0$ such that
\begin{equation}\label{Hajlasz}
\vert u(x)-u(y)\vert\leq d(x,y)^s\left(g_k(x)+g_k(y)\right)
\end{equation} 
for all $k\in\mathbb{Z}$ and for all $x,y\in S\setminus E$ satisfying $2^{k-1}\leq d(x,y)<2^{k}.$ The collection of all fractional $s$-gradient of $u$ is denoted by $\mathbb{D}^s(u).$
\end{definition}
Let $S\subset X$ be a measurable set. For $0<p,q\leq \infty$ and a sequence $\vec{f}=(f_k)_{k\in\mathbb{Z}}$ of measurable functions, we define
\begin{equation*}
\Vert (f_k)_{k\in\mathbb{Z}}\Vert_{L^p(S,l^q)}=\big\Vert \Vert (f_k)_{k\in\mathbb{Z}}\Vert_{l^q} \big\Vert_{L^p(S)}
\end{equation*}
where
\begin{equation*}
\Vert (f_k)_{k\in\mathbb{Z}}\Vert_{l^q}=
\begin{cases}
(\sum_{k\in\mathbb{Z}}\vert f_k\vert^q)^{1/q},& ~\text{when}~0<q<\infty,\\
\sup_{k\in\mathbb{Z}}\vert f_k\vert,& ~\text{when}~q=\infty.
\end{cases}
\end{equation*}
\begin{definition}
Let $S\subset X$ be a measurable set. Let $0<s<\infty$ and $0<p,q\leq\infty.$ The \textit{homogeneous Haj\l asz-Triebel-Lizorkin space} $\dot{M}^s_{p,q}(S)$ consists of measurable functions $u:S\rightarrow [-\infty,\infty],$ for which the (semi)norm
\begin{equation*}
\Vert u\Vert_{\dot{M}^s_{p,q}(S)}=\inf_{\vec{g}\in\mathbb{D}^s(u)}\Vert\vec{g}\Vert_{L^p(S,l^q)}
\end{equation*}
is finite. The (non-homogeneous) \textit{Haj\l asz-Triebel-Lizorkin space} $M^s_{p,q}(S)$ is $\dot{M}^s_{p,q}(S)\cap L^p(S)$ equipped with the norm
\begin{equation*}
\Vert u\Vert_{M^s_{p,q}(S)}=\Vert u\Vert_{L^p(S)}+\Vert u\Vert_{\dot{M}^s_{p,q}(S)}.
\end{equation*}
Note that, when $0<p<1,$ the (semi)norms defined above are actually quasi-(semi)norms.
\end{definition}
We recall that the \textit{generalized Hausdorff $h$-measure} is defined by
\begin{equation*}
\mathcal{H}^h(E)=\limsup_{\delta\rightarrow 0}H_{\delta}^{h}(E),
\end{equation*}
where
\begin{equation*}
H_{\delta}^{h}(E)=\inf\left\{\sum h(r_i) : E\subset\bigcup B(x_i,r_i),~r_i\leq\delta\right\},
\end{equation*}
where $h:[0,\infty)\rightarrow [0,\infty$ is a non-decreasing function such that $\lim_{t\rightarrow 0+}h(t)=h(0)=0,$ see \cite{KKST08}.\\
\indent For the convenience of reader we state here a fundamental covering lemma (for a proof see \cite[2.8.4-6]{Fed69} or \cite[Theorem 1.3.1]{Zie89}).
\begin{lemma}[5B-covering lemma]\label{cover}
Every family $\mathcal{F}$ of balls of uniformly bounded diameter in a metric space $X$ contains a pairwise disjoint subfamily $\mathcal{G}$ such that for every $B\in\mathcal{F}$ there exists $B'\in\mathcal{G}$ with $B\cap B'\neq\emptyset$ and $\diam(B)<2\diam(B').$ In particular, we have that
$$\bigcup_{B\in\mathcal{F}}B\subset\bigcup_{B\in\mathcal{G}}5B.$$
\end{lemma}
\section{Main results}
\noindent Before we go to our main results, let us recall some lemmas from the literature, which will be used to prove the main results.
\begin{lemma}[\cite{HIT16}, Lemma 3.10]\label{HIT-lem}
Let $0<s<1,$ $0<p<\infty,$ $0<q\leq\infty$ and $S\subset X$ be a measurable set. Let $u:X\rightarrow\mathbb{R}$ be a measurable function with $(g_k)\in\mathbb{D}^s(u)$ and let $\phi$ be a bounded $L$-Lipschitz function supported in $S.$ Then sequences $(h_k)_{k\in\mathbb{Z}}$ and $(\rho_k)_{k\in\mathbb{Z}},$ where
$$\rho_k=(g_k\Vert\phi\Vert_{\infty}+2^{k(s-1)}L\vert u\vert)\chi_{\supp \phi}\quad\text{and}\quad h_k=(g_k+2^{sk+2}\vert u\vert)\Vert\phi\Vert_{\infty}\chi_{\supp \phi}$$ 
are fractional $s$-gradients of $u\phi.$ Moreover, if $u\in M^s_{p,q}(S),$ then $u\phi\in M^s_{p,q}(X)$ and $\Vert u\phi\Vert _{M^s_{p,q}(X)}\leq C\Vert u\Vert _{M^s_{p,q}(S)}.$
\end{lemma}
\begin{corollary}[\cite{HIT16}, Corollary 3.12]\label{HIT-cor}
Let $0<s<1,$ $0<p<\infty$ and $0<q\leq\infty.$ Let $\phi:X\rightarrow\mathbb{R}$ be an $L$-Lipschitz function supported in a bounded set $F\subset X.$ Then $\phi\in M^s_{p,q}(X)$ and
\begin{equation*}
\Vert \phi\Vert _{M^s_{p,q}(X)}\leq C(1+\Vert\phi\Vert _{\infty})(1+L^s)\mu (F)^{1/p}
\end{equation*} 
where the constant $C$ depends only on $s$ and $q.$
\end{corollary}
\begin{lemma}[\cite{Nuu16}, Theorem 4.5]\label{poincare-type-lem}
Let $0<\gamma\leq 1/2,$ $0<p<\infty$ and $0<q\leq\infty.$ If $u\in M^s_{p,q}(X),$ then there exist a constant $C>0$ and a sequence $(g_k)_{k\in\mathbb{Z}}$ such that
\begin{equation}\label{poincare-type}
\inf_{c\in\mathbb{R}} m^{\gamma}_{\vert u-c\vert}(B(x,2^{-k}))\leq C2^{-ks}\left(\dashint_{B(x,2^{-k+1})}g_k^p\,d\mu\right)^{1/p}
\end{equation}
for every $x\in X$ and $k\in\mathbb{Z}.$ In fact, for given any $(h_j)_{j\in\mathbb{Z}}\in\mathbb{D}^s(u),$ we can choose
\begin{equation}\label{gradient-construct}
g_k=\left(\sum_{j\geq k-2}2^{(k-j)s'p'}h_j^p\right)^{1/p},
\end{equation}
where $0<s'<s$ and $p'=\min\{1,p\}.$ Moreover, there exists a constant $C_1$ such that
\begin{equation}\label{gradient-compare}
\Vert (g_k)\Vert _{L^p(l^q,X)}\leq C_1 \Vert (h_j)\Vert _{L^p(l^q,X)}.
\end{equation} 
\end{lemma}
\begin{lemma}[\cite{Nuu16}, Theorem 4.8]\label{loc-lipscitz}
Let $0<s<1,$ $0<p,q<\infty$ and $E\subset X$ be a compact set. Then 
\begin{equation*}
\capacity _{M^s_{p,q}}(E)\approx\inf\{\Vert u\Vert ^p_{M^s_{p,q}(X)}:u\in\mathcal{A'}(K)\},
\end{equation*}
where $\mathcal{A'}(K)=\{u\in\mathcal{A}(K): u \,\text{is locally Lipschitz}\}.$
\end{lemma}
\noindent Note that the above lemma is proved in \cite{Nuu16} for Haj\l asz-Besov spaces, but the proof for Haj\l asz-Triebel-Lizorkin spaces is exactly the same.
\begin{lemma}[\cite{HKT17}, Lemma 6.4]\label{HKT-lem}
Let $0<s<\infty$ and $0<p,q\leq\infty.$ Then there exist $c\geq 1$ and $0<r\leq 1$ such that
\begin{equation}\label{cap-eqn}
\capacity _{M^s_{p,q}}\left(\bigcup_{i\in\mathbb{N}}E_i\right)^r\leq c\sum_{i\in\mathbb{N}}\capacity _{M^s_{p,q}}(E_i)^r
\end{equation}
for all sets $E_i\subset X.$ Actually, \eqref{cap-eqn} holds with $r=\min\{1,q/p\}.$
\end{lemma}
Our first result is the following:
\begin{theorem}\label{first-main}
Let $0<s<1,$ $0<p<\infty$ and $0<q\leq \infty.$ Then for any $E\subset X,$ there exists a constant $C$ such that
\begin{equation*}
\capacity_{M^s_{p,q}}(E)^{\theta}\leq C\mathcal{H}^h(E),
\end{equation*}
where $\theta=\min\{1,q/p\}$ and $h(B(x,\rho))=\left(\frac{\mu(B(x,\rho))}{\rho^{sp}}\right)^{\theta}.$
\end{theorem}
\begin{proof}
First let us show that, for $x\in X$ and $0<r\leq 1,$ there exists a constant $C_1$ such that 
\begin{equation}\label{cap-measure}
\capacity_{M^s_{p,q}}(B(x,r))\leq C_1\frac{\mu(B(x,r))}{r^{sp}}.
\end{equation}
Indeed, let us define a function $u:X\rightarrow\mathbb{R}$ by
\begin{equation*}
u(y)=\min\left\{1,\frac{\dist(y,X\setminus B(x,2r))}{r}\right\}.
\end{equation*}
Clearly, $u$ is a $\frac{1}{r}$-Lipschitz function supported in $B(x,2r)$ which takes $1$ in $B(x,r)$ and $0\leq u\leq 1$ on $X.$ Hence, from Corollary \ref{HIT-cor} and the doubling property of $\mu,$ we obtain
\begin{equation*}
\capacity_{M^s_{p,q}}(B(x,r))\leq \Vert u\Vert_{M^s_{p,q}(X)}^p\leq C_1\frac{\mu(B(x,r))}{r^{sp}}.
\end{equation*}
Now, let $\{B(x_i,r_i)\}_{i\in\mathbb{N}}$ be a covering of the set $E$ such that $r_i\leq 1$ for all $i.$ Then from Lemma \ref{HKT-lem}, we get
\begin{eqnarray*}
\capacity_{M^s_{p,q}}(E)^{\theta} &\leq &\left[\capacity_{M^s_{p,q}}(\cup_i B(x_i,r_i))\right]^{\theta}\\
&\leq & c\sum_i \left[\capacity_{M^s_{p,q}} B(x_i,r_i)\right]^{\theta}\\
&\leq & C_2\sum_i\left[\frac{\mu(B(x_i,r_i))}{r_i^{sp}}\right]^{\theta}.
\end{eqnarray*}
After taking the infimum over all the coverings $\{B(x_i,r_i)\}_{i\in\mathbb{N}}$ of $E,$ we conclude the theorem.
\end{proof}
Our main result is the following:
\begin{theorem}\label{second-main}
 Let $x_0\in X,$ $0<R<\infty$ and assume that $B(x_0,8R)\setminus B(x_0,4R)$ is nonempty. If for every compact set $E\subset B(x_0,R),$ $\capacity_{M^s_{p,q}}(E)=0,$ where $0<s<1,$ and $0<p, q<\infty,$ then $\mathcal{H}^h(E)=0$
whenever $$h(B(x,\rho))=\frac{\mu(B(x,\rho))}{\rho ^{sp}}\log^{-p-\epsilon}(1/\rho)$$ for any $\epsilon>0.$
\end{theorem}
\begin{proof}
For simplicity, we assume that $R=2^{-m}$ for some $m\in\mathbb{Z}.$ Let $E\subset B(x_0,2^{-m})$ be a compact set. By Lemma \ref{loc-lipscitz}, we get a locally Lipschitz function $v\in M^s_{p,q}(X)$ which satisfies
\begin{equation*}
\Vert v\Vert ^p_{M^s_{p,q}(X)}\leq C\capacity _{M^s_{p,q}}(E)+\delta
\end{equation*}
for every $\delta>0$ and also $v\geq 1$ on a neighborhood of $E.$ By taking another Lipschitz function $\psi$ which satisfies $\psi=1$ on $B(x_0,2^{-m})$ and $\psi=0$ outside $B(x_0,2^{-m+1}),$ we consider $u=v\psi.$ Note that $u\in M^s_{p,q}(X)$ is Lipschitz and $u\geq 1$ on a neighborhood of $E.$ By Lemma \ref{HIT-lem}, there exists $(g_k)_{k\in\mathbb{Z}}\in\mathbb{D}^s(u)$ such that $g_k$ is supported in $B(x_0,2^{-m+1})$ and
\begin{equation}\label{capacity_gradient}
\Vert (g_k)\Vert ^p_{L^p(l^q,X)}\leq C_1\Vert v\Vert ^p_{M^s_{p,q}(X)}\leq C_2(\capacity _{M^s_{p,q}}(E)+\delta).
\end{equation}
In fact, we can choose $(g_k)$ so that it satisfies \eqref{poincare-type} and the inequalities \eqref{gradient-construct} and \eqref{gradient-compare} ensure that $g_k$ is supported in $B(x_0,2^{-m+1})$ and satisfies \eqref{capacity_gradient}.

Let $x\in E.$ Then
\begin{equation}\label{first}
1\leq u(x)\leq \vert u(x)-m_u^{\gamma}(B(x,2^{-m}))\vert+\vert m_u^{\gamma}(B(x,2^{-m})) \vert.
\end{equation}
Since $u$ is continuous, $x$ is a generalized Lebesgue point of $u.$ Therefore, by using the properties of $\gamma$-median, we obtain
\begin{eqnarray*}
\vert u(x)-m_u^{\gamma}(B(x,2^{-m}))\vert &\leq & \sum_{k\geq m}\vert m_u^{\gamma}(B(x,2^{-k-1}))- m_u^{\gamma}(B(x,2^{-k}))\vert\\
&\leq & \sum_{k\geq m} m_{\vert u- m_u^{\gamma}(B(x,2^{-k}))\vert}^{\gamma}(B(x,2^{-k-1}))\\
&\leq & \sum_{k\geq m}  m_{\vert u- m_u^{\gamma}(B(x,2^{-k}))\vert}^{\gamma/c_d}(B(x,2^{-k}))\\
&\leq & 2\sum_{k\geq m} \inf_{c\in\mathbb{R}} m_{\vert u- c\vert}^{\gamma '}(B(x,2^{-k})),
\end{eqnarray*}
where $\gamma '=\min\{\gamma, \gamma /c_d\}.$
We use Lemma \ref{poincare-type-lem} to conclude
\begin{equation*}
\vert u(x)-m_u^{\gamma}(B(x,2^{-m}))\vert \leq C\sum_{k\geq m}2^{-ks}\left(\dashint_{B(x,2^{-k+1})}g_k^p\,d\mu\right)^{1/p}.
\end{equation*}
Let $y\in B(x,2^{-m})\setminus F$ and $z\in (B(x_0,2^{-m+3})\setminus B(x_0,2^{-m+2}))\setminus F,$ where $F$ is the exceptional set from Definition \ref{s-gradient}. Note that $2^{-m}\leq d(y,z)<2^{-m+4}.$ After defining $g=\max\{g_k:m-4\leq k\leq m-1\},$ we obtain
\begin{eqnarray*}
\vert u(y)\vert=\vert u(y)-u(z)\vert\leq d(y,z)^sg(y)\leq 2^{s(-m+4)}g(y).
\end{eqnarray*}
Again by using the properties of median, we get
\begin{eqnarray*}
\vert m_u^{\gamma}(B(x,2^{-m})) \vert &\leq & m_{2^{s(-m+4)}g}^{\gamma}(B(x,2^{-m}))\\
&\leq & C2^{-ms}m_g^{\gamma}(B(x,2^{-m}))\\
&\leq & C\sum_{k=m-4}^{m-1}2^{-ks}\left(\dashint_{B(x,2^{-k+1})}g_k^p\,d\mu\right)^{1/p}.
\end{eqnarray*}
Using these estimates in \eqref{first}, we obtain
\begin{eqnarray}\label{combine}
1 &\leq & C\sum_{k\geq m-4}2^{-ks}\left(\dashint_{B(x,2^{-k+1})}g_k^p\,d\mu\right)^{1/p} \nonumber\\
&\leq & C\sum_{k\geq m-4}2^{-ks}\left(\dashint_{B(x,2^{-k+1})}\Vert (g_j)_{j\in\mathbb{Z}}\Vert_{l^q}^p\,d\mu\right)^{1/p}.
\end{eqnarray}

\noindent Set $M=C^p(\log 2)^{-p-\epsilon}\big(\sum_{k\geq m-4}\frac{1}{(k-1)^{1+\epsilon/p}}\big)^p,$ which depends only on $\epsilon, p$ and $m.$
If for every $k\geq m-4,$ 
\begin{equation*}
M\int_{B(x,2^{-k+1})}\Vert (g_j)_{j\in\mathbb{Z}}\Vert_{l^q}^p\,d\mu< \frac{\mu(B(x,2^{-k+1}))}{2^{(-k+1)sp}}\log ^{-p-\epsilon}\Big(\frac{1}{2^{-k+1}}\Big),
\end{equation*}
then \eqref{combine} becomes
\begin{equation*}
1<\frac{C}{M^{1/p}}(\log 2)^{-1-\frac{\epsilon}{p}}\sum_{k\geq m-4}(k-1)^{-1-\frac{\epsilon}{p}},
\end{equation*}
which is a contradiction. Therefore, there exists a $k_x\geq m-4$ such that
\begin{equation*}
M\int_{B(x,2^{-k_x+1})}\Vert (g_j)_{j\in\mathbb{Z}}\Vert_{l^q}^p\,d\mu\geq \frac{\mu(B(x,2^{-k_x+1}))}{2^{(-k_x+1)sp}}\log ^{-p-\epsilon}\Big(\frac{1}{2^{-k_x+1}}\Big).
\end{equation*}
By $5B$-covering lemma, pick a collection of disjoint balls $B_i=B(x_i,2^{-k_{x_i}+1})$ such that $E\subset \cup_i5B_i.$ Now, summing over $i,$
\begin{equation*}
M\Vert (g_j)_{j\in\mathbb{Z}} \Vert^p_{L^p(X, l^q)}\geq \sum_i\frac{\mu(B_i)}{\rad(B_i)^{sp}}\log^{-p-\epsilon}\left(\frac{1}{\rad(B_i)}\right),
\end{equation*}
where $M$ is independent of $u,g.$ Since $\capacity_{M^s_{p,q}}(E)=0,$ we see from \eqref{capacity_gradient} that $\Vert (g_j)_{j\in\mathbb{Z}} \Vert^p_{L^p(X, l^q)}$ can be made arbitrarily small and hence $\mathcal{H}^h_{\infty}(E)=0.$  
\end{proof}
\noindent For $s\in (0,1)$ and $p\in (0,\infty),$ the fractional Sobolev space is the space of all $u\in L^p(\mathbb{R}^n)$ such that
\begin{equation*}
\Vert u\Vert_{W^{s,p}(\mathbb{R}^n)}:=\left(\int_{\mathbb{R}^n}\vert u(x)\vert^p\,dx+\int_{\mathbb{R}^n}\int_{\mathbb{R}^n}\frac{\vert u(x)-u(y)\vert^p}{\vert x-y\vert^{n+sp}}\,dx\,dy\right)^{1/p}<\infty.
\end{equation*}
We prove here, using the idea of \cite{Yua15}, the existence of Lebesgue points for $W^{s,p}(\mathbb{R}^n),$ when $s\in (0,1)$ and $p<1.$ When $p\geq 1,$ we know that $W^{s,p}(\mathbb{R}^n)=F^s_{p,p}(\mathbb{R}^n)=M^s_{p,p}(\mathbb{R}^n),$ \cite{KYZ11, Tri83} and hence in this case, Theorem \ref{second-main} together with \cite[Theorem 1.2]{HKT17} give the existence of Lebesgue points. Also see \cite{Kar17} for a similar result in metric spaces.
\begin{theorem}\label{third-main}
Suppose $u\in W^{s,p}(\mathbb{R}^n),$ $s, p\in (0,1).$ Then the generalized Lebesgue points exist outside a set of $\mathcal{H}^{h}$-measure zero, where $h(t)=t^{n-sp}\log^{-p-\epsilon}(1/t)$ for any $\epsilon>0.$
\end{theorem}
\begin{proof}
For $x\in\mathbb{R}^n$, set
\begin{equation*}
g(x)=\sup_{j\in\mathbb{N}}\, 2^{js}\left(\dashint_{B(x,2^{-j})}\vert u(z)-m_u^{\gamma}(B(x,2^{-j}))\vert^p\, dz\right)^{1/p}.
\end{equation*}
For each $j,$ using the properties of medians, we have
\begin{eqnarray*}
& & \dashint_{B(x,2^{-j})}\vert u(z)-m_u^{\gamma}(B(x,2^{-j}))\vert^p\, dz\\
&\leq & \dashint_{B(x,2^{-j})}\vert u(z)-u(x)\vert^p\, dz+\vert u(x)-m_u^{\gamma}(B(x,2^{-j}))\vert^p\\
&\lesssim & \dashint_{B(x,2^{-j})}\vert u(z)-u(x)\vert^p\, dz\\
&\lesssim & 2^{-jsp}\int_{B(x,2^{-j})}\frac{\vert u(z)-u(x)\vert^ p}{2^{-j(n+sp)}}\, dz\\
&\lesssim & 2^{-jsp}\int_{B(x,2^{-j})}\frac{\vert u(z)-u(x)\vert^ p}{\vert z-x\vert^{n+sp}}\, dz.
\end{eqnarray*}
Hence
\begin{equation*}
g(x)\lesssim \left(\int_{\mathbb{R}^n}\frac{\vert u(z)-u(x)\vert^ p}{\vert z-x\vert^{n+sp}}\, dz\right)^{1/p},
\end{equation*}
which implies that $\vert\vert g\vert\vert_{L^p(\mathbb{R}^n)}\lesssim\vert\vert u\vert\vert_{W^{s,p}(\mathbb{R}^n)}<\infty.$\\
For each $j,$ choose a ball $B(x_j,2^{-j-2})\subset B(x,2^{-j})\setminus B(x,2^{-j-1}).$ For each $z\in B(x_j,2^{-j-2}),$ we have $B(x,2^{-j-1})\subset B(x,2^{-j})\subset B(z,2^{-j+2})$ and hence, by using the properties of medians,
\begin{eqnarray*}
& &\vert m_u^{\gamma}(B(x,2^{-j}))-m_u^{\gamma}(B(x,2^{-j-1}))\vert \\
&\leq & \vert m_u^{\gamma}(B(x,2^{-j}))-m_u^{\gamma}(B(z,2^{-j+2}))\vert + \vert m_u^{\gamma}(B(x,2^{-j-1}))-m_u^{\gamma}(B(z,2^{-j+2}))\vert\\
&\lesssim & \sum_{l=j-1}^{j}\left(\dashint_{B(x,2^{-l})}\vert u(\omega)-m_u^{\gamma}(B(z,2^{-j+2}))\vert^p\, d\omega\right)^{1/p}\\
&\lesssim & \left(\dashint_{B(z,2^{-j+2})}\vert u(\omega)-m_u^{\gamma}(B(z,2^{-j+2}))\vert^p\, d\omega\right)^{1/p}\\
&\lesssim & 2^{-js}g(z),
\end{eqnarray*}
which implies that
\begin{eqnarray*}
\vert m_u^{\gamma}(B(x,2^{-j}))-m_u^{\gamma}(B(x,2^{-j-1}))\vert &\lesssim & \frac{2^{-js}}{\vert B(x_j,2^{-j-2})\vert^{1/p}}\left(\int_{B(x_j,2^{-j-2})}[g(z)]^p\,dz\right)^{\frac{1}{p}}\\
&\lesssim & 2^{-j(s-n/p)}\left(\int_{B(x,2^{-j})}[g(z)]^p\,dz\right)^{\frac{1}{p}}.
\end{eqnarray*}
Therefore, for $l,m\in\mathbb{N}$ with $l>m,$
\begin{equation}\label{cauchy}
\vert m_u^{\gamma}(B(x,2^{-l}))-m_u^{\gamma}(B(x,2^{-m}))\vert\lesssim \sum_{j=m}^{l-1}2^{-j(s-n/p)}\left(\int_{B(x,2^{-j})}[g(z)]^p\,dz\right)^{\frac{1}{p}}.
\end{equation}
Let $h_1(t)=t^{n-sp}\log^{-p-\epsilon/2}(1/t).$ If $\int_{B(x,r)}[g(z)]^p\,dz\leq Ch_1(r)$ for some constant $C$ and for all sufficiently small $0<r<1/5,$ then $\{m_u^{\gamma}(B(x,2^{-j}))\}_j$ is a Cauchy sequence, by \eqref{cauchy}. We would like to show that the set
\begin{multline*}
E_\epsilon=\bigg\{x\in\mathbb{R}^n :\text{there exists arbitrarily small}~0<r_x<\frac{1}{5}~\text{such that}\\ \int_{B(x,r_x)}[g(z)]^p\,dz\geq Ch_1(r_x)\bigg\}
\end{multline*}
has Hausdorff $h$-measure zero. Towards this end, let $0<\delta<1/5.$ By using $5B$-covering lemma we obtain a pairwise disjoint family $\mathcal{G}$ consisting of balls as above such that $$E_{\epsilon}\subset\bigcup_{B\in\mathcal{G}}5B,$$
where $\diam(B)<2\delta$ for $B\in\mathcal{G}.$ Then
\begin{eqnarray*}
\mathcal{H}_{10\delta}^{h_1}(E_{\epsilon}) & \leq & C\sum_{B\in\mathcal{G}} h_1\left(\rad(B)\right)\\
& \leq & C_1\sum_{B\in\mathcal{G}} \int_{B}[g(z)]^p\,dz\\
& \leq & C_1\int_{\bigcup\limits_{B\in\mathcal{G}}B}[g(z)]^p\,dz<\infty.
\end{eqnarray*}
It follows that $\mathcal{H}^{h_1}(E_{\epsilon})<\infty$ and hence $\mathcal{H}^{h}(E_{\epsilon})=0.$ This shows that $\lim_{r\rightarrow 0}m_u^{\gamma}(B(x,r))$ exists outside $E_{\epsilon}$ with $\mathcal{H}^{h}(E_{\epsilon})=0,$ where $h(t)=t^{n-sp}\log^{-p-\epsilon}(1/t).$
\end{proof}

\def\bibname{References}
\bibliography{Triebel_Lizorkin}
\bibliographystyle{plain}
\end{document}